\documentclass[a4paper,11pt]{amsart}

\usepackage{amssymb, amscd, latexsym}
\usepackage{amsmath}
\usepackage{amsthm}
\usepackage{graphicx}
\usepackage[all]{xy}
\usepackage{mathrsfs}
\usepackage{inputenc}
\usepackage{bigstrut}
\usepackage{multirow}
\usepackage{array}
\usepackage{ctable}
\usepackage{booktabs}
\usepackage{tabularx}
\usepackage{fancyhdr}
\usepackage{float}
\usepackage{arydshln}
\usepackage{hyperref}
\hypersetup{
pdftitle={Research paper},
pdfauthor={Ali Akbar Yazdan Pour},
pdfsubject={Condidates for non-zero Betti numbers of monomial ideals},
pdfcreator={Ali Akbar Yazdan Pour},
colorlinks,
linkcolor=blue,
citecolor=magenta,
anchorcolor=red,
bookmarksopen,
urlcolor=red,
filecolor=red,
pdfpagetransition={Wipe}
}

\theoremstyle{plain}
\newtheorem{thm}{Theorem}[section]
\newtheorem{cor}[thm]{Corollary}

\newtheorem{prop}[thm]{Proposition}

\theoremstyle{definition}

\newtheorem{ex}[thm]{Example}

\theoremstyle{remark}
\newtheorem{rem}{Remark}

\def\cocoa{{\hbox{\rm C\kern-.13em o\kern-.07em C\kern-.13em o\kern-.15em A}}}
\def\Tor#1#2#3#4{{\rm Tor}_{#1}^{#2}\,({#3},{#4})}
\def\reg{{\rm reg}}
\def\ZZ{{\mathbb Z}}

\begin{document}

\title{candidates for non-zero betti numbers of monomial ideals}
\author[a. a. yazdan pour]{{ali akbar} {yazdan pour}}
\address{Department of Mathematics\\ Institute for Advanced Studies in Basic Sciences (IASBS)\\ P.O.BOX: 45195-1159 \\ Zanjan, Iran}
\email{yazdan@iasbs.ac.ir}

\begin{abstract}
Let $I$ be a monomial ideal in the polynomial ring $S$ generated by elements of degree at most $d$. In this paper, it is shown that, if the $i$-th syzygy of $I$ has no element of degrees $j, \ldots, j+(d-1)$ (where $j \geq i+d$), then $(i+1)$-syzygy of $I$ does not have any element of degree $j+d$. Then we give several applications of this result, including an alternative proof for Green-Lazarsfeld index of the edge ideals of graphs as well as an alternative proof for Fr\"oberg's theorem on classification of square-free monomial ideals generated in degree two with linear resolution. Among all, we describe the possible indices $i, j$ for which $I$ may have non-zero Betti numbers $\beta_{i,j}$.
\end{abstract}
\maketitle

\section*{introduction}
Monomial ideals are probably the most important objects in combinatorial commutative algebra and their graded minimal free resolution is a wide area of research in this subject. It is known that, the graded Betti numbers of a monomial ideal and its polarization are the same. Polarization is a technique to transform a monomial ideal to a square-free monomial ideal. On the other hand, the Stanley-Reisner theory, corresponds a simplicial complex to a given square-free monomial ideal and Hochster formula \cite{Hochster} enables us to find the Betti numbers of square-free monomial ideals in purely combinatorial method; though determining the Betti numbers of these ideals is not as easy as using only this formula. In fact, computational methods in commutative algebra and computer algebra systems such as \textsc{Singular} \cite{Singular} and \cocoa~\cite{CoCoA} use the technique of the Gr\"obner basis to find the Betti numbers of an ideal and not the Hochster formula. So it seems that the combinatorics of the monomial ideals is still mysterious. 

On the other hand, many numerical invariants of a graded ideal (such as, projective dimension, Castelnuovo-Mumford regularity, etc), focus on indices for which we have a non-zero Betti numbers, since determining the exact value of Betti numbers seems not to be accessible in many cases. This paper is intended to describe a (sharp) bound for an interval for which, we should search for the non-zero Betti numbers among it. Indeed, we show that, for a given monomial ideal $I$ generated in degree at most $d$, if the $i$-th syzygy of $I$ contains no elements of degrees $j, j+1, \ldots, j+(d-1)$ (where $j \geq i+d$), then in the next syzygy, there is no element of degree $j+d$ (Theorem~\ref{Growth of Beti numbers-Cor1}). Among all other interesting corollaries, this enables us to describe the possible shape of the Betti diagram of a monomial ideal. A similar result in the case of edge ideal of a graph, can be found in \cite{Oscar} by O.~Fern{\'a}ndez-Ramos and P.~Gimenez. This paper is organized as follows: In Section~\ref{Preliminaries}, we recall some basic notions and prerequisites on commutative algebra and combinatorics. Section~\ref{Section main Theorem} contains the main theorem of this paper. To prove the main theorem, first we obtain the result in the case of square-free monomial ideals and then, we use the technique of polarization, to obtain the result in the case of monomial ideals. Our method here, is to observe on the vanishing of reduced homology in certain degrees. The last section is devoted to some applications of the main theorem. In this section, we state several applications of the main theorem and we pursue with generalizations of known results on this subject.

Recently, the author found that Theorem~\ref{Growth of Beti numbers-Cor1} has been independently proved in \cite{Varbaro} to justify their algorithm for computing Betti numbers and Corollary~\ref{Growth of Beti numbers-Cor2}(i) appeared in \cite{Herzog-Srinivasan} as a partial answer to subadditivity problem. Independently the author proved these results and presented them in \cite{Y12}.% and \cite[in 2014]{Y14}.

\section{preliminaries} \label{Preliminaries}
In this section, we recall basic notations and prerequisites of this paper which will be used later. This section is divided into three parts. In the first part, the notion of simplicial complexes and Mayer-Vietoris sequence will be described. Then, we review quickly the definition of Betti numbers and regularity. Finally, we recall basic notations in graph theory; the Fr\"oberg's theorem is also stated in this part. Throughout, $S=K[x_1, \ldots, x_n]$ denotes the polynomial ring over a field $K$ with the standard grading (i.e. $\deg(x_i)=1$).

\subsection*{Simplicial complexes}
%\begin{defn}
A \textit{simplicial complex} $\Delta$ over a set of vertices $V=\{ v_{1}, \ldots, v_{n} \}$, is a collection of subsets of $V$, with the property that:
\begin{itemize}
\item[(a)] $\{ v_{i} \} \in \Delta $, for all $i$;
\item[(b)] if $F\in \Delta$, then all subsets of $F$ are also in $\Delta$ (including the empty set).
\end{itemize}
An element of $\Delta$ is called a \textit{face }of $\Delta$ while a \emph{non-face} of $\Delta$ is a subset $F$ of $V$ with $F \notin \Delta$. We denote by $\mathcal{N}(\Delta)$, the set of all minimal non-faces of $\Delta$. The maximal faces of $\Delta$ with respect to inclusion are called \textit{facets} of $\Delta$.

Let $\mathcal{F}(\Delta) =\{F_{1}, \ldots, F_{q}\}$ be the facet set of $\Delta$. It is clear that $\mathcal{F}(\Delta)$ determines $\Delta$ completely and so we write $\Delta = \langle F_{1}, \ldots, F_{q} \rangle$. A simplicial complex $\Gamma$ is called a \textit{subcomplex} of $\Delta$, if $\Gamma \subset \Delta$.
%\end{defn}

Given a simplicial complex $\Delta$ over $n$ vertices labeled $v_{1}, \ldots, v_{n}$. For $F \subset \{v_{1}, \ldots, v_{n} \}$, we set:
\begin{equation*}
\textbf{x}_F=\prod\limits_{v_i \in F}{x_i}.
\end{equation*}
with $\textbf{x}_\varnothing = 1$.

The \emph{non-face ideal} or the \emph{Stanley-Reisner ideal} of $\Delta$, denoted by $I_\Delta$, is the ideal of $S$ generated by square-free monomials $\textbf{x}_F$ where $F \in \mathcal{N}(\Delta)$. Also we call $K[\Delta]:=S/I_\Delta$ the \emph{Stanley-Reisner ring} of $\Delta$.

Let $\tilde{H}_i \left( \Delta; K \right)$ denotes the reduced homology of a simplicial complex $\Delta$ with the coefficients in the field $K$. If $\Delta$ is a simplicial complex and $\Delta_1$ and $\Delta_2$ are subcomplexes of $\Delta$, then there is an exact sequence:
\begin{align} \label{Reduced Mayer-Vietoris sequence}
\cdots \to \tilde{H}_j(\Delta_1 \cap \Delta_2 ; K) \to \tilde{H}_j(\Delta_1; K) \oplus \tilde{H}_j(\Delta_2; K) \to  \tilde{H}_j(\Delta_1 \cup \Delta_2 ; K) \to \nonumber \\
\to  \tilde{H}_{j-1}(\Delta_1 \cap \Delta_2; K) \to \cdots
\end{align}
with all coefficients in $K$, called the \textit{reduced Mayer-Vietoris sequence} of $\Delta_1$ and $\Delta_2$ (see \cite[Proposition 5.1.8]{HerzogHibi} for more details).

\subsection*{Betti numbers and Hochester formula}
Let $M = \oplus_{i\in \ZZ}M_i$ be a finitely generated graded $S$-module and
$$
 \cdots \to F_i \stackrel{\varphi_i}{\longrightarrow} F_{i-1} \to \cdots \to F_1 \stackrel{\varphi_1}{\longrightarrow} F_0 \stackrel{\varphi_0}{\longrightarrow} M \stackrel{\varphi_{-1}}{\longrightarrow} 0
$$
a graded minimal free resolution of $M$ with $F_i = \oplus_j S(-j)^{\beta^K_{i,j}(M)}$ for all $i$. The graded $S$-module $\mathrm{Syz}_i (M) :=\mathrm{Ker} (\varphi_{i-1})$ ($i \geq 0$) is called the $i$-th syzygy module of $M$ and the numbers $\beta_{i,j}^K(M) = \dim_K \mbox{Tor}^S_i(K,M)_j$ are called the \textit{graded Betti numbers} of $M$. Also, the \textit{projective dimension} of $M$,  $\mbox{projdim}(M)$, is defined to be the number
$$
\mbox{projdim}(M) = \sup\{i \colon \quad \mbox{Tor}^S_i(K,M) \neq 0\}.
$$
Note that $\beta^K_{i,j}(M)$ counts the elements of degree $j$ in a minimal generator of $i$-th syzygy.

For every $i \in \mathbb{N} \cup \{0\}$, one defines:
$$t^S_i(M) = \max \{j \colon \quad \beta^K_{i,j}(M) \neq 0 \}$$
and $t^S_i(M)= - \infty$, if it happens that $\Tor{i}{S}{K}{M}= 0$. The \textit{Castelnuovo-Mumford regularity} (or simply \textit{regularity}) of $M$, $\reg(M)$, is given by:
$$\reg (M)= \sup \{t^S_i(M)-i \colon \quad i \in \ZZ \}.$$
It is worth to say that, by graded version of Hilbert's syzygy theorem \cite[Corollary 19.7]{Eisenbud}, every finitely generated graded $S$-module has a minimal free resolution of length at most $n$, the number of variables of $S$. Hence its regularity and projective dimension are finite.

The \textit{initial degree}, ${\rm indeg} (M)$, is given by:
$${\rm indeg}\, (M)=\inf \{i \colon \quad M_i \neq 0 \}.$$
A finitely generated graded $S$-module $M$ has a \textit{$d$-linear resolution}, if its regularity is equal to $d = {\rm indeg} (M)$.

There exists a beautiful formula due to Hochster, that describes the Betti numbers of a square-free monomial ideal $I$ in terms of the dimension of reduced homologies of $\Delta$, when $I = I_\Delta$.
\begin{thm}[{Hochster formula \cite[Theorem 5.1]{Hochster}}] \label{Hochster Formula}
Let $\Delta$ be a simplicial complex on $[n] = \{1, \ldots, n\}$, $K$ a field. Then,
\begin{align*}
\beta^K_{i,j}(I_\Delta) = \sum\limits_{\substack{W \subset [n] \\ |W|=j}}{\dim_K \tilde{H}_{j-i-2}(\Delta_W; K)},
\end{align*}
where $\Delta_W$ is the simplicial complex on the vertex set $W$, whose faces are $F \in \Delta$  with $F \subseteq W$.
\end{thm}

\subsection*{Graph and edge ideal}
A simple \textit{graph} $G = (V, E)$ consists of a finite set $V$ of \textit{vertices} and a collection $E$ of subsets of $V$, called \textit{edges}, such that every edge of $G$ is a pair $\left\{ u, v \right\}$ for some $u,v \in V$ with $u \neq v$. 

Given a graph $G$ on the vertex set $[n]=\{1, \ldots, n\}$, the \textit{edge ideal}, $I(G) \subset S$, is the square-free monomial ideal defined as follows:
$$I({G}) = \left( x_ix_j \colon \qquad \{i,j\} \in E(G) \right).$$
It is clear that, there exists a one to one correspondence between square-free monomial ideals generated in degree $2$ in $S$ and the edge ideals of graphs. Finding algebraic properties of a square-free monomial ideal generated in degree $2$ in terms of combinatorial properties of the associated graph, is a very active area of research in commutative algebra. One of the most beautiful result in this subject is the Fr\"oberg's theorem on characterization of square-free monomial ideals generated in degree $2$ with linear resolution. To state this theorem, let us first recall some elementary concepts in the graph theory.

For a  graph $G$, we denote by $V(G)$ and $E(G)$ the set of vertices and edges of $G$ respectively. A graph $K$ is \textit{complete} if $E(K)$ contains all $2$-subsets of $V(K)$. A graph $H$ is called a \textit{subgraph} of $G$ if $V':=V(H) \subseteq V(G)$ and $E':=E(H) \subseteq E(G)$. Furthermore, $H$ is called the \textit{induced subgraph} of $G$, if $E' = \left\{ \{u,v\} \in E(G) \colon \quad u,v \in V(H) \right\}$; in this case, we write $H=G_{V'}$. A subset $W$ of $V$ is called a \textit{clique} in $G$, if $G_W$ is a complete graph. The set of cliques of $G$, forms a simplicial complex, which is called \textit{clique complex} of $G$ and is denoted by $\Delta(G)$. One may easily check that $I_{\Delta(G)} = I \left( \bar{G} \right)$, where $\bar{G}$ is a graph with $V \left( \bar{G} \right) = V\left( {G} \right)$ and $E \left( \bar{G} \right) = \left\{ \{u,v\} \colon \quad u, v \in V(G) \text{ and } \{u,v\} \notin E(G) \right\}$.

A \textit{cycle} of length $q$ is a graph $C=(V, E)$ with $V=\{1, \ldots, q\}$ and
$$E=\big\{ \{i_1, i_2\}, \{i_2, i_3\}, \ldots, \{i_{q-1}, i_q\}, \{i_q, i_1 \} \big\},$$
where $i_j \neq i_k$, if $j \neq k$. An induced subgraph which is also a cycle is called an \textit{induced cycle}, and a graph $G$ is said to be \textit{chordal} if it has no induced cycle of length $>3$.

With the above definitions, we are now  ready to state Fr\"oberg's theorem. In Section~\ref{Applications}, we will recover this theorem by the results of this paper.

\begin{thm}[{Fr\"oberg's Theorem \cite{Froberg}}] \label{Froberg Theorem}
Let $G$ be a graph on the vertex set $[n]$ and $I=I(G)$ be the associated edge ideal of $G$. The ideal $I$ has a linear resolution if and only if $\bar{G}$ is a chordal graph.
\end{thm}

\section{On the growth of the range of non-zero Betti numbers} \label{Section main Theorem}
%\section{On the Growth of the degree of syzygies}

As it is mentioned in the introduction of this paper, some numerical invariants of modules depend only on indices for which we have a non-zero Betti numbers. So it is of great importance to find a method to bound the intervals for which the Betti numbers may not be zero. The main theorem of this section provides a sharp bound for this conclusion. In fact, our aim in this section is to prove the following theorem. We will see several applications of this theorem in the next section.

In the following, for a monomial ideal $I \subset S$, we denote by $\mathcal{G} \left( I \right)$, the unique set of monomial minimal generators of $I$.
\begin{thm}[Main Theorem] \label{Growth of Beti numbers-Cor1}
Let $I$ be a monomial ideal in the polynomial ring $S$ and $d := \max \{ \deg (u) \colon \; u \in \mathcal{G}(I) \}$. If $i\geq 0$ and $j \geq i+d$ be non-negative integers such that,
	$$\beta^K_{i,j}(I) = \beta^K_{i,j+1}(I) = \cdots = \beta^K_{i,j+(d-1)}(I) = 0,$$
then, $\beta^K_{i+1,j+d}(I) = 0$.
\end{thm}

To prove this theorem, let us first consider the square-free case. In this case, we need the following two propositions to prove the main theorem.
\begin{prop} \label{Homology of equidimensional ideals}
	Let $\Delta$ be a simplicial complex on the vertex set $[n]$ and $A=\{a_1, \ldots, a_d\}$ be a $d$-subset of $[n]$ such that $A \notin \Delta$. Let $I:=I_\Delta \subset K[x_1, \ldots, x_n]$ be the Stanley-Reisner ideal of $\Delta$ and $i \geq 0, j\geq i+d$ be non-negative integers with,
	$$\beta^K_{i,j}(I) = \beta^K_{i,j+1}(I) = \cdots = \beta^K_{i,j+(d-1)}(I) = 0.$$
	If $W$ is a subset of $[n]$ with $|W|=j+d$ and $A \subset W$, then:
	\begin{equation} \label{EQ Homology of Equidimensinal ideals}
	\tilde{H}_{j-i+(d-t-3)} \left( \bigcup\limits_{i={j_0}}^{d} \Delta_{W \setminus \{a_1, \ldots, a_t, a_i\}} ;K \right) =0
	\end{equation}
	for all $t, j_0$, with $0 \leq t \leq d-1$ and $t< j_0 \leq d$.
\end{prop}

\begin{proof}
	We use induction on $t':=d-t$ to obtain (\ref{EQ Homology of Equidimensinal ideals}). Note that, $1 \leq t' \leq d$ and our assumption $\beta^K_{i,j}(I)=0$, implies that:
	\begin{equation*}
	\tilde{H}_{j-i-2}\left( \Delta_{W \setminus \{a_1, \ldots, a_d\}};K\right) =0.
	\end{equation*}
	This proves (\ref{EQ Homology of Equidimensinal ideals}), for $t'=1$. Let $1<t' \leq d-1$ and (\ref{EQ Homology of Equidimensinal ideals}) holds for $t'$. We must show that, (\ref{EQ Homology of Equidimensinal ideals}) is true for $t'+1$. That is, we must show that:
	\begin{equation} \label{EQ2 Homology of Equidimensinal ideals}
	\tilde{H}_{j-i+(d-t-2)} \left( \bigcup\limits_{i={j_0}}^{d} \Delta_{W \setminus \{a_1, \ldots, a_{t-1}, a_i\}} ;K \right) =0, \qquad \text{for all $j_0$, with } t \leq j_0 \leq d. 
	\end{equation}
	
	By our assumption, $\beta^K_{i,j+(d-t)}(I)=0$. So that,
	\begin{equation*} 
	\tilde{H}_{j-i+(d-t-2)} \left( \Delta_{W \setminus \{a_1, \ldots, a_{t-1}, a_d\}} ;K \right) =0.
	\end{equation*}
	This gives (\ref{EQ2 Homology of Equidimensinal ideals}), for $j_0=d$. Now, fix $t \leq j_0 <d$. We observe that:
	\begin{align}
	\bigcup\limits_{i={j_0}}^{d} \Delta_{W \setminus \{a_1, \ldots, a_{t-1}, a_i\}} = \left( \Delta_{W \setminus \{a_1, \ldots, a_{t-1}, a_{j_0}\}} \right) \cup \left( \bigcup\limits_{i={j_0+1}}^{d} \Delta_{W \setminus \{a_1, \ldots, a_{t-1}, a_i\}} \right), \nonumber\\
	\left( \Delta_{W \setminus \{a_1, \ldots, a_{t-1}, a_{j_0}\}} \right) \cap \left( \bigcup\limits_{i={j_0+1}}^{d} \Delta_{W \setminus \{a_1, \ldots, a_{t-1}, a_i\}} \right) = \bigcup\limits_{i={j_0}+1}^{d} \Delta_{W \setminus \{a_1, \ldots, a_{t-1}, a_{j_0}, a_i\}} \nonumber
	\end{align}
	Our induction hypothesis, implies that:
	\begin{equation} \label{Local EQ a.4}
	\tilde{H}_{j-i+(d-t-3)} \left( \bigcup\limits_{i={j_0}+1}^{d} \Delta_{W \setminus \{a_1, \ldots, a_{t-1}, a_{j_0}, a_i\}}; K \right) =0.
	\end{equation}
	Also, 
	\begin{equation} \label{Local EQ a.5}
	\tilde{H}_{j-i+(d-t-2)} \left( \Delta_{W \setminus \{a_1, \ldots, a_{t-1}, a_{j_0}\}}; K \right) =0;
	\end{equation}
	for, $\beta^K_{i, j+(d-t)}(I) =0$. Using (\ref{Local EQ a.4}), (\ref{Local EQ a.5}) and Mayer-Vietoris long exact sequence (\ref{Reduced Mayer-Vietoris sequence}), we conclude that:
	\begin{align*}
	\tilde{H}_{j-i+(d-t-2)} \left( \bigcup\limits_{i={j_0}}^{d} \Delta_{W \setminus \{a_1, \ldots, a_{t-1}, a_i\}} ;K \right) \cong \tilde{H}_{j-i+(d-t-2)} \left( \bigcup\limits_{i={j_0}+1}^{d} \Delta_{W \setminus \{a_1, \ldots, a_{t-1}, a_i\}} ;K \right).
	\end{align*}
	If $j_0+1=d$, then (\ref{EQ2 Homology of Equidimensinal ideals}) holds, by our previous discussion. Otherwise, by repeating this process, after a finite number of steps, we get:
	\begin{align*}
	\tilde{H}_{j-i+(d-t-2)} \left( \bigcup\limits_{i={j_0}}^{d} \Delta_{W \setminus \{a_1, \ldots, a_{t-1}, a_i\}} ;K \right) \cong \tilde{H}_{j-i+(d-t-2)} \left(\Delta_{W \setminus \{a_1, \ldots, a_{t-1}, a_d\}} ;K \right).
	\end{align*}
	The homology in the right hand side is zero, because $\beta^K_{i, j+(d-t)}(I)=0$. This completes the proof.
\end{proof}

\begin{prop} \label{Growth of Beti numbers}
	Let $\Delta$ be a simplicial complex on the vertex set $[n]$ and $I=I_\Delta \subset S$ be the Stanley-Reisner ideal of $\Delta$. Let $d = \max \{ \deg (u) \colon \; u \in \mathcal{G}(I) \}$. If $i\geq 0$ and $j \geq i+d$ be non-negative integers such that,
	$$\beta^K_{i,j}(I) = \beta^K_{i,j+1}(I) = \cdots = \beta^K_{i,j+(d-1)}(I) = 0,$$
	then, $\beta^K_{i+1,j+d}(I) = 0$.
\end{prop}

\begin{proof}
	If $\beta^K_{i+1,j+d}(I) \neq 0$, then by Theorem~\ref{Hochster Formula}, there exists $W \subset [n]$, with $|W|=j+d$ and
	\begin{equation} \label{Local EQ a.6}
	\tilde{H}_{j-i+(d-3)}\left( \Delta_W; K \right) \neq 0.
	\end{equation}
	In particular, (\ref{Local EQ a.6}) implies that, $I_{\Delta_W} \neq 0$. Hence, by our choice of $d$, there exist elements $a_1, \ldots, a_d \in W$, such that $\{a_1, \ldots, a_d\} \notin \Delta$. One can easily check that:
	\begin{equation*}
	\Delta_W = \bigcup\limits_{i=1}^{d} \Delta_{W \setminus \{a_i\}}.
	\end{equation*}
	Hence, applying Proposition~\ref{Homology of equidimensional ideals}, for $t=0, j_0=1$, we get:
	\begin{equation*}
	\tilde{H}_{j-i+(d-3)}\left( \Delta_W; K \right) = \tilde{H}_{j-i+(d-3)}\left( \bigcup\limits_{i=1}^{d} \Delta_{W \setminus \{a_i\}}; K \right) = 0.
	\end{equation*}
	This contradicts (\ref{Local EQ a.6}).
\end{proof}

Now we are ready to prove Theorem~\ref{Growth of Beti numbers-Cor1}.
\begin{proof}[Proof of Theorem~\ref{Growth of Beti numbers-Cor1}]

Using the technique of polarization (See for example \cite[Section 1.6]{HerzogHibi}), Proposition~\ref{Growth of Beti numbers} can be easily extended to monomial ideals as well. To be more precise, let $J:= \mathcal{P}(I) $ be the polarization of $I$. Then, by \cite[Corollary 1.6.3(a)]{HerzogHibi}, we have $\beta_{i,j}^K(I) = \beta_{i,j}^K(J)$ for all $i$ and $j$. Since $J$ is a square-free monomial ideal, the conclusion follows immediately from Proposition~\ref{Growth of Beti numbers}.
\end{proof}

\begin{rem}
Let $G$ be a simple graph on the vertex set $[n]$ and $I=I \left( {G} \right)$ be the edge ideal of ${G}$. It is shown in \cite[Theorem 1.8]{Oscar} that if $i \geq 0$ and $j \geq i+2$ are integers such that $\beta_{i,j}^K (I) = \beta_{i,j+1}^K (I) =0$, then $\beta_{i+1,j+2}^K (I) =0$. Clearly, this result (and its consequences) is a special case of Theorem~\ref{Growth of Beti numbers-Cor1}. In the next section, we propose more applications of Theorem~\ref{Growth of Beti numbers-Cor1}.
\end{rem}

\section{Some applications} \label{Applications}

\subsection*{A criterion for linear resolution}
Let $I= I_\Delta$ be a square-free monomial ideal generated in degree (at most) $d$. In order to show that $I$ has a $d$-linear resolution, we must show that  the reduced homology vanishes in all degrees$\neq d-2$. The following proposition, shows that, this problem concerns to vanishing of the reduced homology observed only in degrees $d-1, \ldots, \min \left\{n-1, 2d-3 \right\}$.

\begin{prop} \label{vanishing of reduced homology}
Let $\Delta$ be a simplicial complex on the vertex set $[n]$. Let $I=I_\Delta \subset S$ be the Stanley-Reisner ideal generated in degree $d$. Then the following are equivalent:
\begin{itemize}
\item[\rm (a)] $\tilde{H}_i \left(\Delta_W; K \right) = 0$, for all $W \subseteq [n]$ and all $i= d-1, \ldots, \min \left\{n-1, 2d-3 \right\}$;
\item[\rm (b)] $\beta^K_{i,j}(I) = 0$, for all $j>i+d$;
\item[\rm (c)] $I$ has a $d$-linear resolution.
\end{itemize}
\end{prop}

\begin{proof}
The implications (b)$\Rightarrow$(c) is by definition while (c)$\Rightarrow$(a) is a direct consequence of Hochster formula. We need only to prove (a)$\Rightarrow$(b). We use induction on $i$ to obtain (b). Note that (b) holds for $i=0$, since $I$ is generated in degree $d$. Let $i>0$ and (b) holds for $i-1$. We consider two cases for $j$.
\begin{itemize}
\item[] \textbf{Case 1.} If $i+d < j < i+2d$, then by our assumption, 
\begin{align*}
\tilde{H}_{j-i-2} \left(\Delta_W; K \right) = 0, \qquad \text{for all } W \subseteq [n].
\end{align*}
Hence, by Hochster formula, we obtain
\begin{align*}
\beta^K_{i,j}(I_\Delta) = \sum\limits_{\substack{W \subset [n] \\ |W|=j}}{\dim_K \tilde{H}_{j-i-2}(\Delta_W; K)} =0.
\end{align*}

\item[] \textbf{Case 2.} If $j \geq i+2d$, then by our induction hypothesis, we have
$$\beta^K_{i-1,j-d} \left( I \right) = \cdots = \beta^K_{i-1,j-1} \left( I \right) = 0.$$
Now, Theorem~\ref{Growth of Beti numbers-Cor1} implies that $\beta^K_{i,j} \left( I \right) = 0$.
\end{itemize}
\end{proof}

\begin{ex}[Fr\"oberg's Theorem] \label{example for froberg}
In this example, we use Proposition~\ref{vanishing of reduced homology}, to give a very short proof for Fr\"oberg's Theorem (Theorem~\ref{Froberg Theorem}). Let $\Delta= \Delta(G)$ be the clique complex of $G$. Then by Proposition~\ref{vanishing of reduced homology}, $I_\Delta = I \left( \bar{G} \right)$ has a linear resolution if and only if $\tilde{H}_1 \left( \Delta_W; K \right) =0$, for all $W \subseteq [n]$. Since $\tilde{H}_1 \left( \Delta; K \right)$ is generated by induced cycles in $G$, so having linear resolution is equivalent to saying that $G$ is a chordal graph. 
\end{ex}

\subsection*{$N_{d,p}$-property and Green-Lazarsfeld index}
The (Green-Lazarsfeld) index of a graded ideal measures the number of linear steps in the graded minimal free resolution of the ideal. Let $I \neq 0$ be a homogeneous ideal in $S=K[x_1, \ldots, x_n]$. The ideal $I$ is said to satisfy ${\rm N}_{d,p}$-property ($p>0$), if
\begin{equation*}
\beta^K_{i, i+j} (I) =0, \quad \text{for all } i < p, \text{ and } j>d.
\end{equation*}
For a non-zero homogeneous ideal $I \subset S$, it is clear that, $I$ has a $d$-linear resolution if and only if, $I$ satisfies ${\rm N}_{d,p}$-property, for all $p > 0$.

\begin{table}[!htp] 
\centering
\begin{tabular}{|c|c c c c c c c|}
\hline
%\backslashbox{$j$}{$i$}
& $0$ & $1$ & $\cdots$ & $i$ & $\cdots$ & $p-1$ & $p$ \\
\hline
$1$ & & & & & & \multicolumn{1}{c:}{} & \\
$2$ & & & & & & \multicolumn{1}{c:}{} & \\
$\vdots$ & & & & & & \multicolumn{1}{c:}{} & \\
$j$ & & & & $\beta^K_{i, i+j}$ & & \multicolumn{1}{c:}{} & \\
$\vdots$ & & & & & & \multicolumn{1}{c:}{} & \\
$d$ & & & & & &  \multicolumn{1}{c:}{} & \\
\cline{2-7}
$d+1$ & $0$ & $0$ & $\cdots$ & $0$ & $\cdots$ & \multicolumn{1}{c|}{$0$} & $\times$ \\
 & $0$ & $0$ & $\cdots$ & $0$ & $\cdots$ & \multicolumn{1}{c|}{$0$} & $\times$ \\
$\vdots$ & $\vdots$ & $\vdots$ & & $\vdots$ &  & \multicolumn{1}{c|}{$\vdots$} & $\vdots$ \\
$\reg(I)$ & $0$ & $0$ & $\cdots$ & $0$ & $\cdots$ & \multicolumn{1}{c|}{$0$} & $\times$ \\
\cline{1-8}
\end{tabular}
\vspace*{2mm}
\caption{Betti diagram of an ideal with ${\rm N}_{d,p}$ property}
\label{Betti diagram of ideal with N property}
\end{table}

Let $I$ be a non-zero homogeneous ideal and $d=\mathrm{indeg}(I)$. The number,
$$\mathrm{index}(I) = \sup \left\{p \colon \qquad I \text{ satisfies $N_{d,p}$-property} \right\}$$
is called the \textit{index} of $I$. So that $\mathrm{index}(I) = \infty$ if and only if $I$ has a $d$-linear resolution. It is, in general, very difficult to determine the precise value of the Green-Lazarsfeld index. Important conjectures, such as Green's conjecture \cite[Chapter 9]{Eisenbud}, predicts the value of this invariant for certain families of varieties. However, the remarkable result of Eisenbud et al., enables us to find the index of ideals associated to graphs \cite[Theorem 2.1]{EGHP}. In the following we recover this result by our own method.

\begin{thm} \label{Eisenbud theorem}
Let $G$ be a simple graph on the vertex set $[n]$ and $I=I(\bar{G}) \subset S$ be the corresponding ideal. Then the following are equivalent:
\begin{itemize}
\item[\rm (a)] $I$ satisfies $N_{2,p}$-property;
\item[\rm (b)] every induced cycle in $G$ has length $\geq p+3$.
\end{itemize}
In particular, 
\begin{equation} \label{Index of a graph}
\mathrm{index}(I) = \inf \left\{ |C|-3 \colon \quad C \text{ is an induced cycle in } G \text{ of length}>3 \right\}
\end{equation}
\end{thm}

\begin{proof}
(a)$\Rightarrow$(b). Assume on the contrary that $G$ has an induced cycle $C$, of length $j < p+3$ and let $W$ be the set of vertices of $C$. Then $|W|=j$ and $\Delta_W$ coincides with $C$. In particular, $\tilde{H}_1 \left( \Delta_W; K \right) \neq 0$. Now by Hochster formula, we have
\[
\beta^K_{j-3,j} (I) \geq \dim_K \tilde{H}_1 \left( \Delta_W; K \right) >0.
\]
This contradicts the fact that $I$ satisfies $N_{2,p}$-property.

(b)$\Rightarrow$(a). Assume that $I$ does not satisfy $N_{2,p}$-property and take the minimal  $i<p$ and $j>i+2$, such that $\beta^K_{i,j} (I) \neq 0$. If $j>i+3$, then the minimality of $i$ implies that $\beta^K_{i-1, j-2} \left( I \right) = \beta^K_{i-1,j-1}  \left( I \right) =0$. So, $\beta^K_{i,j} \left( I \right) =0$ by Theorem~\ref{Growth of Beti numbers-Cor1}. This contradicts to our choice of $i$ and $j$. Hence, $j=i+3$, and Hochster formula implies that $\tilde{H}_1 \left( \Delta_W; K \right) \neq 0$ for some $W \subset [n]$ with $|W|= j$. Again the minimality of $j$ implies that $\Delta_W$ coincides with an induced cycle in $G$; and  our assumption implies that $|W| \geq p+3$. Thus, $i\geq p$ which contradicts to our choice of $i$.

It is now immediately deduced from equivalence of (a) and (b) that, the index of $I$ is obtained from (\ref{Index of a graph}).
\end{proof}

\begin{ex}[Fr\"oberg's Theorem again]
In Example~\ref{example for froberg},  we have seen a short proof for Fr\"oberg's theorem on the classification of monomial ideals with $2$-linear resolution. This theorem can be  obtained as a result of Theorem~\ref{Eisenbud theorem} as well. Indeed, let $I=I(\bar{G})$ be the edge ideal of $\bar{G}$. Then $I$ has a $2$-linear resolution, if and only if $\mathrm{index} (I) = \infty$ and by equation~(\ref{Index of a graph}) this is the case, if and only if $G$ does not have any induced cycle of length $>3$. The later is equivalent to say that $G$ is chordal.
\end{ex}

\subsection*{Non-zero Betti numbers}
Let $M$ be a finitely generated graded $S$-module. The Hilbert syzygy's theorem, guarantees  that, there are only finite numbers $i, j$ such that $\beta_{i,j}^K(M)$ is non-zero. In general, it is so hard to determine the exact value of the Betti numbers. However, since many invariants of graded module depends only on the indices of non-zero Betti numbers, it is natural to ask, for which indices $i$ and $j$, we have a non-zero Betti number $\beta_{i,j}$. One of the known (easy) result on this subject is the following proposition.

\begin{prop}[{\cite[Prop. 1.9]{Eisenbud2}}] \label{simlar result with lower bound}
Let $M$ be a finitely generated $S$-module. If $d$ is an integer such that $\beta^K_{i,j} (M) = 0$ for all $j < d$, then $\beta_{i+1,j+1}^K (M) = 0$ for all $j < d$.
\end{prop}

It follows from Proposition~\ref{simlar result with lower bound} that, if $I$ is a monomial ideal with $c:= \mathrm{indeg}(I)$ and $\beta_{i,j}^K (I) \neq 0$, then $j \geq i+c$. We proceed with a partial generalization of this fact (see Corollary~\ref{Growth of Beti numbers-Cor3}); but first we need the following consequence of Theorem~\ref{Growth of Beti numbers-Cor1}.

\begin{cor} \label{Growth of Beti numbers-Cor2}
Let $I \subset S:=K[x_1, \ldots, x_n]$ be a non-zero monomial ideal, $\rho={\rm projdim}(I)$ and $d = \max \{ \deg (u) \colon \quad u \in \mathcal{G}(I) \}$. Then,
	\begin{itemize}
		\item[\rm (i)] $t_{i+1}^S(I) \leq t_i^S(I) +d$.
		\item[\rm (ii)] $\reg(I) \leq \rho(d-1)+d$.
	\end{itemize}
\end{cor}

\begin{proof}
(i) By the definition of $t_i^S(I)$, we have $\beta_{i,j}^K(I) = 0$, for all $j>t_i^S(I)$. Hence, by Corollary~\ref{Growth of Beti numbers-Cor1} we have, $\beta_{i+1,j}^K(I) = 0$, for all $j>d+t_i^S(I)$. This proves (i).
	
(ii) We note that, $d=t_0^S(I)$, by our choice of $d$. Hence (i) implies that, $t_i^S(I) \leq (i+1)d$. So that,
	\begin{align*}
	\reg(I)&= \max \{t_i^S(I)-i \colon \qquad 0 \leq i \leq \rho \} \\
	& \leq \max \{(i+1)d -i \colon \quad 0 \leq i \leq \rho \} \leq \rho(d-1)+d.
	\end{align*}
\end{proof}

\begin{rem}
It is immediate from definition of Betti numbers that, $\beta^K_{i+1,j} \left( S/I \right) = \beta^K_{i,j} \left( I \right)$ holds for any homogeneous ideal $I$ in $S$ and for all $i, j$. Hence, for a monomial ideal $I \subset S$, we may restate Corollary~\ref{Growth of Beti numbers-Cor2}(i) as:
\begin{equation} \label{Subadditivity problem}
t_{i+1}^S \left( {S}/{I} \right) \leq t_i^S \left( {S}/{I} \right) + t_1^S \left( {S}/{I} \right).
\end{equation}
We consider the following generalization of (\ref{Subadditivity problem}):
\begin{equation} \label{Subadditivity problem 2}
t_{i+j}^S \left( {S}/{I} \right) \leq t_i^S \left( {S}/{I} \right) + t_j^S \left( {S}/{I} \right).
\end{equation}
for every $i$ and $j$. No counterexample is known to the validity of (\ref{Subadditivity problem 2}) for monomial ideals, while for arbitrary homogeneous ideal, (\ref{Subadditivity problem 2}) is not true \cite[Sec. 2]{Conca}. However, the inequality in (\ref{Subadditivity problem}) shows that (\ref{Subadditivity problem 2}) holds for $j=1$ and for monomial ideals.

In \cite[Corollary 1.9]{Oscar}, the authors obtained (\ref{Subadditivity problem}), for monomial ideals generated in degree $2$ and in \cite[Corollary 4]{Herzog-Srinivasan}, the authors proved the same inequality as in (\ref{Subadditivity problem}) for arbitrary monomial ideal, by an independent method. Another interesting case for which (\ref{Subadditivity problem 2}) comes true, is when $i+j = \# \left\{ \text{variables of } S \right\}$ and  $S/I$ is of Krull dimension at most $1$ \cite[Theorem 4.1]{ECU}.
\end{rem}

\begin{cor}  \label{Growth of Beti numbers-Cor3}
Let $I \subset S:=K[x_1, \ldots, x_n]$ be a monomial ideal, $c={\rm indeg}(I)$ and $d = \max \left\{ \deg (u) \colon \quad u \in \mathcal{G}(I) \right\}$.
\begin{itemize}
	\item[\rm (i)] If $\beta^K_{i,j}(I) \neq 0$, then $i+c \leq j \leq d(i+1)$.
	\item[\rm (ii)] If $I$ is square-free monomial ideal and $\beta^K_{i,j}(I) \neq 0$, then $i+c \leq j \leq \min \{n, d(i+1) \}$.
\end{itemize}
\end{cor}

\begin{proof}
First assume that, $I$ is a square-free monomial ideal with Stanley-Reisner complex $\Delta$. It follows from \cite[Proposition 3.1(i)]{mar2} that $\tilde{H}_i(\Delta_W; K) =0$, for all $i<c-2$ and $W \subset [n]$. In view of Theorem~\ref{Hochster Formula}, one has $i+c \leq j \leq n$, if $\beta_{i,j}^K(I) \neq 0$. Also, as we have seen in the proof of Corollary~\ref{Growth of Beti numbers-Cor2}(ii), we have $t_i^S(I) \leq (i+1)d$. Therefore, $j \leq (i+1)d$, if $\beta_{i,j}^K(I) \neq 0$.
	
Statement (ii) follows from the above discussion. Note that, polarization changes the number of variables in the polynomial ring while it does not change the graded Betti numbers. So, the same argument leads us to (i).
\end{proof}

\begin{rem} \rm
Let $c$ and $d$ be as in Corollary~\ref{Growth of Beti numbers-Cor3}. In \cite[Lemma 2.2]{Katzman} and \cite[Theorem 3.2.3]{HaVanTuyl1}, the authors used Taylor resolution to obtain Corollary~\ref{Growth of Beti numbers-Cor3}(i) for $c=d=2$ and in \cite[Theorem 2.6]{HaVanTuyl2}, the authors used the same method to obtain Corollary~\ref{Growth of Beti numbers-Cor3}(ii) for square-free monomial ideals that are generated by elements in the same degree (i.e. when, $c=d$).
\end{rem}

\begin{ex}
In this example, we show that the bounds in Corollary~\ref{Growth of Beti numbers-Cor3}(i) for possible non-zero Betti numbers, are sharp enough. To do this, let $G$ be a cycle of length $4$ and $I=I \left( \bar{G} \right) \subset S=K[x_1, x_2, x_3, x_4]$. One may verify that, the graded minimal free resolution of $I$ is
\[
0 \to S(-4) \to S^{2}(-2) \to I \to 0.
\]
Hence, the degree $j$ in $\beta_{1,4}^K(I) \neq 0$ obtains its upper bound in Corollary~\ref{Growth of Beti numbers-Cor3}. This example also, shows that the bound for the regularity in Corollary~\ref{Growth of Beti numbers-Cor2}(ii) is also sharp.
\end{ex}


\begin{thebibliography}{99}

\bibitem{CoCoA}
J.~Abbott, A.~Bigatti, and G.~Lagorio, ``{CoCoA-5}: a system for doing
  {C}omputations in {C}ommutative {A}lgebra.'' Available at
  \url{http://cocoa.dima.unige.it}.

\bibitem{Conca}
L.~L. Avramov, A.~Conca, and S.~B. Iyengar, ``Subadditivity of syzygies of
  koszul algebras,'' {\em Mathematische Annalen}, vol.~361, no.~1,
  pp.~511--534, 2014.

\bibitem{ref2}
J.~B\"{o}hm and S.~A.~Papadakis, ``On the structure of Stanley--Reisner rings associated to cyclic polytopes,'' \emph{Osaka J. Math.}, vol.~49, no.~1, pp.~81--100, 2012.

\bibitem{Singular}
W.~Decker, G.-M. Greuel, G.~Pfister, and H.~Sch\"onemann, ``{\sc Singular}
  {4-0-2} --- {A} computer algebra system for polynomial computations.''
  \url{http://www.singular.uni-kl.de}, 2015.

\bibitem{Eisenbud}
D.~Eisenbud, \textit{Commutative Algebra, with a View toward Algebraic Geometry}, Springer-Verlag, Berlin-Heidelberg-New York, 1995.

\bibitem{Eisenbud2}
D. Eisenbud, \textit{The Geometry of Syzygies: A second course in commutative algebra and algebraic geometry}, in: GTM 229, Springer, New York, 1995.

\bibitem{EGHP}
D.~Eisenbud, M.~Green, K.~Hulek and S.~Popescu, ``Restricting linear syzygies: algebra and geometry,'' {\em Compos. Math.} vol.~141, no.~6, pp. 1460--1478, 2005.

\bibitem{ECU}
D.~Eisenbud, C.~Huneke and B.~Ulrich, ``The regularity of Tor and graded Betti numbers,'' {\em American Journal of Mathematics}, vol.~128, no.~3, pp.~573--605, 2006.

\bibitem{Froberg}
R.~Fr\"{o}berg, ``On Stanley--Reisner rings,'' in: \emph{Topics in algebra}, Banach Center Publications, 26 Part 2, pp.~57--70, 1990.

\bibitem{HaVanTuyl1}
H.~T. H{\`a} and A. Van Tuyl, ``Resolutions of square-free monomial ideals via facet ideals: a survey'', \emph{Contemporary Mathematics}, vol.~448, pp.~91--117, 2007.

\bibitem{HaVanTuyl2}
H.~T. H{\`a} and A.~Van Tuyl, ``Monomial ideals, edge ideals of hypergraphs, and
  their graded betti numbers,'' {\em Journal of Algebraic Combinatorics},
  vol.~27, no.~2, pp.~215--245, 2007.

\bibitem{HerzogHibi}
J.~Herzog and T.~Hibi, \textit{Monomial Ideals}, in: GTM 260, Springer, London, 2010.

\bibitem{Katzman}
M.~Katzman, ``Characteristic-independence of betti numbers of graph ideals,''
  {\em Journal of Combinatorial Theory, Series A}, vol.~113, no.~3, pp.~435 --
  454, 2006.

\bibitem{mar2}
M. Morales, A. A. Yazdan Pour and R. Zaare-Nahandi, ``Regularity and Free Resolution of Ideals which are Minimal to $d$-linearity'', \emph{Math. Scand.}, vol.~118, no. 2, pp.~161--182, 2016.

\bibitem{Hochster}
M.~Hochster, \textit{Cohen-Macaulay rings, combinatorics, and simplicial complexes}, in: Ring Theory, II, Proc. Second Conf., Univ. Oklahoma, Norman, Okla., 1975, in: Lecture Notes in Pure and Appl. Math., vol. 26, Dekker, New York, 1977, pp. 171--223.

\bibitem{Herzog-Srinivasan}
J.~Herzog, H.~Srinivasan, ``A note on the subadditivity problem for maximal shifts in free resolutions,'' to appear in \emph{MSRI Proc.}, \href{http://arxiv.org/abs/1303.6214}{\texttt{arXiv:math/1303.6214v1}}

\bibitem{Oscar}
O.~Fern{\'a}ndez-Ramos and P.~Gimenez, ``Regularity $3$ in edge ideals associated to bipartite graphs,'' \emph{Journal of Algebraic Combinatorics}, vol.~39, no.~4, pp. 919--937, 2014. 
  
\bibitem{McCullough}
J.~McCullough, ``A polynomial bound on the regularity of an ideal in terms of half the syzygies,'' \emph{Math. Res. Lett.}, vol.~19, no.~3, pp.~555--565, 2012.

\bibitem{Schenzel}
P.~Schenzel, ``\"{U}ber die freien Aufl\"{o}sungen extremaler Cohen--Macaulay Ringe,'' \emph{Journal of Algebra}, vol.~64, pp.~93--101, 1980.

\bibitem{Stanley}
R.~P.~Stanley, \textit{Combinatorics and Commutative Algebra}, Second Ed., Birkh\"auser, 1996.

\bibitem{ref1}
N.~Terai and T.~Hibi, ``Computation of Betti numbers of monomial ideals
  associated with cyclic polytopes,'' {\em Discrete and Computational Geometry},
  vol.~15, no.~3, pp.~287--295.

\bibitem{Varbaro}
M.~L.~Torrente and M.~Varbaro, ``An alternative algorithm for computing the Betti table of a monomial ideal,'' ArXiv e-prints 2015, \href{http://arxiv.org/abs/1507.01183}{\texttt{arXiv:math/1507.01183}}

\bibitem{Y12}
A.~A.~Yazdan Pour, \textit{Two Results on the Regularity of Monomial Ideals}, in: Proceeding of $10^{\rm th}$ Seminar on Commutative Algebra and Related Topics (In honor of Prof. Hossein Zakeri), 18--19 December  2013, School of Mathematics (IPM), Thehran, IRAN. Available at:\\
\url{math.ipm.ac.ir/conferences/2013/10th\_commalg/SlideShow/YazdanPour.pdf}

\end{thebibliography}
\end{document}